\documentclass[12pt]{article}
\usepackage{amsmath,amssymb,amsthm}
\def\var{\mbox{Var}}
\def\bbe{\mathbb E}
\def\bbr{\mathbb R}

\theoremstyle{plain}
\newtheorem{lem}{Lemma}[section]
\newtheorem{cor}[lem]{Corollary}

\newtheorem{thm}[lem]{Theorem}
\newtheorem{rmk}[lem]{Remark}

\theoremstyle{definition}
\numberwithin{equation}{section}

\begin{document}

\title{Iterated Jackknives And Two-Sided Variance Inequalities}
\author{O.~Bousquet\thanks{Google Research, Brain Team, 8002 Z\"urich, Switzerland.} \; C.~Houdr\'e\thanks{Georgia Institute of Technology, 
School of Mathematics, Atlanta, Georgia, 30332-0160.} 
\thanks{Partially supported by the 
grant \# 524678 from the Simons Foundation and by a Simons Fellowship grant 
\# 267336.  This material is also based, in part, upon work supported by the 
National Science Foundation under Grant No.1440140, while this author was in residence at the Mathematical 
Sciences Research Institute in Berkeley, California, during the Fall semester of 2017. }
}

\maketitle
\begin{abstract}
We consider the variance of a function of $n$ independent random variables and provide new inequalities which, in particular, extend  previous results obtained for symmetric functions  
in the i.i.d.~setting.  
For instance, we obtain various upper and lower variance bounds 
based on iterated jackknives statistics that can be considered as  
generalizations of the Efron-Stein inequality.
\end{abstract}

\section{Introduction}
The properties of functions of $n$ independent random variables, 
and in particular the estimation of their moments from the moments of their increments (i.e., when replacing a random variable by an independent copy) have been thoroughly studied (see, e.g., \cite{BLM} for a 
comprehensive overview). We focus here on the variance and consider how to refine and generalize 
known extensions of the Efron-Stein inequality in the non-symmetric, non-iid case.

First, let us review some of the existing results. 
Let $X_1,X_2,\dots, X_n$ be iid random variables and let $S:\bbr^n\to \bbr$
be a statistic of interest which is symmetric, i.e., invariant under any 
permutation of its arguments, and square integrable.  The (original) Efron-Stein
inequality \cite{ES}, states that the jackknife estimates of variance is
biased upwards, i.e., denoting by $\tilde X$ an independent copy of $X_1,\dots, X_n$,
and setting $S_i=S(X_1,\dots, X_{i-1},X_{i+1},\dots, X_n,\tilde X)$, 
$i=1,\dots, n$, and $S_{n+1}=S$, then
\begin{equation}\label{eq1.1}
\var\ S\le\bbe  J_1,
\end{equation}
where
\begin{align}\label{eq1.2}
J_1=\sum^{n+1}_{i=1}(S_i-\bar S)^2=\frac1{(n+1)}\mathop{\sum\sum}\limits_{1\le i<j\le n+1} (S_i-S_j)^2, 
\end{align}
and $\bar S=\sum^{n+1}_{i=1}S_i/(n+1)$.
Beyond the original framework, the inequality \eqref{eq1.1} has seen many extensions
and generalizations with different proofs which are well described in \cite{BLM}, whose notation we 
essentially adopt 
and to which we refer for a more complete bibliography and many instances of applications. 
Let us just say that
\eqref{eq1.1} can be seen as the ``well known" tensorization property of the 
variance which asserts that if $X_1,X_2,\dots, X_n$ are independent
random variables with $X_i\sim \mu_i$, then
\begin{equation}\label{eq1.3}
\var_{\mu^n} S\le \bbe_{\mu^n}\sum^n_{i=1}\var_{\mu_i} S,
\end{equation}
where $\bbe_{\mu^n}$ and $\var_{\mu^n}$ are respectively the expectation and
variance with respect to $\mu^n:=\mu_1\otimes\cdots\otimes\mu_n$, the joint law of $X_1,X_2,\dots, X_n$, 
while $\var_{\mu_i} S$ is the variance of $S$ with respect to $\mu_i$, the 
law of $X_i$.  In fact, if for each $i=1,2,\dots, n$, $\tilde X_i\sim\tilde\mu_i$ is
an independent copy of $X_i$, then \eqref{eq1.3} can be rewritten as
\begin{align}\label{eq1.4}
\var_{\mu^n} S&\le\frac12\bbe_{\mu^n}\sum^n_{i=1}\bbe_{\mu_i\otimes\tilde\mu_i}
(S-S_i)^2\nonumber\\
&=\frac12\bbe_{\mu^n}\sum^n_{i=1}\bbe_{\tilde\mu_i}(S-S_i)^2,
\end{align}
where $S_i=S(X_1,\dots, X_{i-1},\tilde X_i, X_{i+1},\dots, X_n)$.

Neither \eqref{eq1.1} nor \eqref{eq1.4}, whose proof can be obtained, for example, by 
induction, require $S$ to be symmetric.  In case $S$ is symmetric, and the random variables are 
identically distributed, the right-hand side of \eqref{eq1.4} becomes $n\bbe_{\mu^n\otimes\tilde\mu_1}
(S-S_1)^2/2$ while, via \eqref{eq1.2}, the right-hand side of \eqref{eq1.1} becomes
$\binom{n+1}{2}\bbe(S_1-S_2)^2/(n+1)= n\bbe(S_1-S_2)^2/2$, and
\eqref{eq1.4} and \eqref{eq1.1} are equal.

Since the jackknife estimate of variance is biased upwards, it is natural to try
to estimate the bias $\bbe J_1-\var\ S$, and such an attempt is already presented in 
\cite{H} via the ``iterated jackknives".   Let us recall what was meant there:
Resampling the jackknife statistics, introduce for any $k=2,\dots, n$, the
iterated jackknives $J_2,J_3,\dots, J_n$, leading to both upper and lower 
bounds on $\var\ S$, showing, in particular, that 
\begin{equation}\label{eq1.5}
\frac12\bbe J_2-\frac16\bbe J_3\le \bbe J_1-\var\ S
\le\frac12\bbe J_2.
\end{equation}

In \cite{H}, the inequalities \eqref{eq1.1} and \eqref{eq1.5} were viewed as statistical
versions of generalized (multivariate) Gaussian Poincar\'e inequalities previously obtained in \cite{HK}.  
Indeed, setting $\nabla S : =(S-S_1, S- S_2, \dots, S-S_n)$, then $\bbe J_1 = \bbe\|\nabla S\|^2$.  
If instead of looking at the vector of first differences, one looks at second and third ones,  
then the corresponding norms will lead to (1.5).  
Throughout the years, it was asked whether or not an inequality such as \eqref{eq1.5} 
would have a general version 
and a positive answer had been informally given.  
The aim of the present note is to provide a synthetic proof of these,  
removing the iid and symmetry assumptions in \eqref{eq1.5} and 
its  generalizations, leading to generic inequalities.   
This could be useful,  as these dormant inequalities 
seem to have 
found, in recent times, some new life, e.g., see \cite{M}, \cite{BGS}, \cite{MP}.

\section{Iterated Jackknife Bounds}

Throughout and unless otherwise  noted, $X_1,\dots, X_n$ are independent random variables 
and $S:\bbr^n\to \bbr$ is a 
Borel function such that $\bbe S^2(X_1, \dots, X_n) < +\infty$.  
Next, and if $S$ is short for $S(X_1, \dots, X_n)$, let, for any $i=1, \dots, n$, 
$\bbe^{(i)}$ denote the conditional expectation with respect to 
the $\sigma$-field generated  by $X_1, \dots, X_{i-1}, X_{i+1}, \dots, X_n$.  Hence,  
\begin{align}\label{defcond}
\bbe^{(i)} S&:= \bbe (S\mid X_1,\dots, X_{i-1},X_{i+1},\dots, X_n) \nonumber\\
&\;= \int_{-\infty}^{+\infty}S(X_1,\dots, X_{i-1},x_i, X_{i+1},\dots, X_n)\mu_i(dx_i),
\end{align}
where $\mu_i$ is the law of $X_i$.  By convention, $\bbe^{(0)}$ is the identity operator and so 
$\bbe^{(0)} S = S$.  Iterating the above, it is clear that 
\begin{align}\label{commu}
\bbe^{(i)}\bbe^{(j)} S= \bbe^{(j)}\bbe^{(i)} S &= \bbe(S\mid X_1,\dots, X_{i-1}, X_{i+1},\dots, X_{j-1}, X_{j+1},\dots, X_n) \\
&:=\bbe^{(i,j)} S = \bbe^{(j,i)} S, \nonumber
\end{align}
for any $i, j=1, \dots, n$ and that for $i= 0, 1, \dots, n$, 
$$\bbe^{(i)}\bbe^{(0)} S= \bbe^{(0)}\bbe^{(i)} S := \bbe^{(i,0)} S = \bbe^{(0,i)} S= \bbe^{(i)} S. $$   
Next, let  
$$\var^{(i)}S := \bbe^{(i)}( S - \bbe^{(i)} S)^2 = \bbe^{(i)} S^2 -(\bbe^{(i)} S)^2,$$
$i=0,1, \dots, n$, and  for any $i, j = 0, 1, \dots, n$, set   
\begin{equation}\label{firstit}
\var^{(i,j)}S := \bbe^{(i)}\var^{(j)}S  - \var^{(j)}\bbe^{(i)} S = \var^{(j,i)}S \ge 0.  
\end{equation}
where, above,  the rightmost equality follows from the commutativity property of 
the conditional expectations, as given in  \eqref{commu}, 
while the inequality follows from convexity, and more precisely 
from the conditional Jensen's inequality.

At this point we also note that although $\var^{(i)}$ is the conditional variance 
with respect to the $\sigma$-field generated  by $X_1, \dots, X_{i-1}, X_{i+1}, \dots, X_n$, 
$\var^{(i,j)}$ is not the conditional variance 
with respect to the $\sigma$-field generated  by 
$X_1, \dots, X_{i-1}, X_{i+1}, \dots, X_{j-1}, X_{j+1}, \dots, X_n$.  
Indeed, 
\begin{equation}\label{firstdecomp}
\var^{(i,j)}S = \bbe^{(i,j)}(S- \bbe^{(i,j)}S)^2 - \var^{(i)}\bbe^{(j)} S - \var^{(j)}\bbe^{(i)} S.  
\end{equation}
Further iterating, for $i_1, i_2, \dots, i_k \in \{0, 1, 2, \dots, n\}$, then $\bbe^{(i_1)}\cdots\bbe^{(i_k)}:=\bbe^{(i_1,i_2\dots,i_k)}$ 
is uniquely defined, i.e., the order in which the indices are taken is 
irrelevant, in particular $\bbe^{(1, 2, \dots, n)}S = \bbe S$.  
Still, iterating, set 
\begin{equation}\label{generalit}
\var^{(i_1,i_2, \dots, i_k)}S := \bbe^{(i_1)}\var^{(i_2, \dots, i_{k})}S  - 
\var^{(i_2, \dots, i_{k})}\bbe^{(i_1)}S,   
\end{equation}
where again, above, the order in which the 
indices $i_1,i_2, \dots, i_k \in \{0, 1, 2, \dots, n\}$ are taken is irrelevant, and 
further, by convexity, \eqref{generalit} is non--negative, i.e., 
$$\var^{(i_1,i_2, \dots, i_k)}S \ge 0.$$ 

Another set of identities, more in line with \cite{H}, is also easily obtained via iterated differences, namely,   
$\bbe\var^{(i)}S = \bbe\left(S-S_i\right)^2/2$, and iterating, 
\begin{equation}\label{diffiterees}
\bbe \var^{(i_1,i_2, \dots, i_k)}S =\frac{1}{2^k}\bbe\left(\left(S-S_{i_1}\right)_{i_2, \dots, i_k}\right)^2.
\end{equation}

With the help of the above definitions, and in view of \cite{H},  let us now introduce the iterated 
jackknives, 
$$J_k := \sum_{1\le i_1\ne i_2\cdots\ne i_k\le n}
\var^{(i_1,\dots, i_k)}S = k!\sum_{1\le i_1<i_2<\cdots<i_k \le n} \var^{(i_1,\dots, i_k)}S .$$

Clearly, $J_1 = \sum_{i=1}^n \var^{(i)} S$ and in view of \eqref{defcond}, \eqref{eq1.3} 
can just be rewritten as:   

\begin{equation}\label{jack1}
\var S \le \bbe\sum^n_{i=1}\var^{(i)} S = \bbe J_1.    
\end{equation}

Still in view of the results of \cite{H}, we  now intend to prove:

\begin{thm}\label{jackit}  For any $p=1, 2, \dots, [n/2]$, 
\begin{equation}\label{ineq}
\sum^{2p}_{k=1}\frac{(-1)^{k+1}}{k!}\bbe J_k \le {\rm Var}\; S \le\sum^{2p-1}_{k=1} 
\frac{(-1)^{k+1}}{k!} \bbe J_k, 
\end{equation}
and 
\begin{equation}\label{varexp}
{\rm Var}\; S = \sum^n_{k=1} \frac{(-1)^{k+1}}{k!}\bbe J_k.
\end{equation}
\end{thm}

\begin{proof}
The proof of \eqref{varexp} is a simple decomposition/induction, while 
that of \eqref{ineq} further uses convexity.  For $k=1, 2, \dots, n$,  
let 
$$R_k=\sum_{1\le i_1<\cdots <i_k\le n} \var^{(i_1,\dots, i_k)}
(\bbe^{(1,\dots, i_1-1)}S),$$
with the understanding that for $i=1$, $\bbe^{(1,i-1)}S= \bbe^{(0)}S = S$.   Then,  first note that, 
\begin{align}
\bbe R_1 &= \bbe \sum_{i_1=1}^n \left((\bbe^{(1,\dots, i_1-1)}S)^2 
- (\bbe^{(1,\dots, i_1)}S)^2\right) \nonumber \\
&= \bbe(S^2 - (\bbe S)^2)  = \var S.   
\end{align}
Notice further that for $2\le k \le n-1$, 
\begin{align}\label{recur}
\bbe R_k&= \bbe \sum_{1\le i_1<\cdots < i_k\le n}\var^{(i_1,\dots, i_k)}
(\bbe^{(1,\dots, i_1-1)}S) \nonumber\\
&=\bbe \sum_{1\le i_1<\cdots < i_k\le n}\left( \var^{(i_2,\dots, i_k)}
(\bbe^{(1,\dots, i_1-1)}S)-\var^{(i_2,\dots, i_k)} (\bbe^{(1,\dots, i_1)}S)\right) \nonumber \\
&=\bbe \sum_{1<i_2<\cdots < i_k\le n}\sum_{i_1=1}^{i_2-1}\left( \var^{(i_2,\dots, i_k)}
(\bbe^{(1,\dots, i_1-1)}S)-\var^{(i_2,\dots, i_k)} (\bbe^{(1,\dots, i_1)}S)\right) \nonumber \\
&=\bbe \sum_{1\le i_2<\cdots < i_k\le n}\left(\var^{(i_2,\dots, i_k)} S-
\var^{(i_2,\dots, i_k)}(\bbe^{(1,\dots, i_2-1)} S)\right) \nonumber \\
&=\frac{\bbe J_{k-1}}{(k-1)!}-\bbe R_{k-1}.
\end{align}
Finally, it is clear that, $R_n = \var^{(1, \dots, n)} S = J_n/{n!}$, and so 
$ n!\bbe R_n = \bbe J_n.  $
Combining the last three identities, gives \eqref{varexp}.  
To obtain \eqref{ineq}, note first that by convexity and for any 
$1\le i_1<i_2<\cdots <i_k\le n$, 
\begin{equation}\label{ConvIne}
\bbe^{(1,\dots, i_1-1)} \var^{(i_1,\dots, i_k)} S\ge
\var^{(i_1,\dots, i_k)} (\bbe^{(1,\dots, i_1-1)}S).  
\end{equation}
Hence, taking expectation and summing gives $\bbe J_k\ge k!\bbe R_k$, which 
when combined with \eqref{recur} finishes the proof.  
\end{proof}

\begin{rmk} 
(i)  In case $S$ is symmetric, i.e., invariant under any permutation of its arguments, 
$J_k = n(n-1)\dots (n-k+1)\var^{(1, \dots, k)}S$, then 
$\bbe J_k = n(n-1)\dots (n-k+1)\bbe\var^{(1, \dots, k)}S$, and \eqref{varexp} 
and \eqref{ineq} precisely recover corresponding results in \cite{H}.  

(ii)  The inequalities \eqref{ineq} can be viewed as martingale inequalities.  

(iii)  As in \cite{BLM} or \cite{BGS}, one could also rewrite \eqref{ineq} using only 
the positive or negative parts of the involved quantities.  

(iv)  It is natural to wonder whether or not the above 
inequalities have $\bf\Phi$-entropic versions; this will be 
explored and presented elsewhere.  
\end{rmk}

Let us now further refine \eqref{ineq} providing, in particular, a non-trivial non-negative lower  bound on the 
bias $\bbe J_1 - \var\, S$ improving upon  \eqref{eq1.5}.  
To do so, denote by $\overline{(i_1,\dots, i_k)}$ the complement of the indices $(i_1,\dots, i_k)$ 
(i.e., the ordered sequence of elements of the set $\{1,\ldots,n\}\backslash \{i_1,\dots, i_k\}$), 
and introduce, for $k\ge 1$,  the following quantities:
$$K_k := k!\sum_{1\le i_1<i_2<\cdots<i_k \le n} \var^{(i_1,\dots, i_k)} {\bbe}^{\overline{(i_1,\dots, i_k)}}S .$$

It is clear that by Jensen's inequality and the convexity of $\var^{(i_1,\dots, i_k)}$ we have
\[
\bbe K_k \le \bbe J_k\,.
\]

\begin{thm}\label{jackit2}  For any $p=1, 2, \dots, [n/2]$, 
\begin{equation}\label{ineq2}
\sum^{2p}_{k=1}\frac{(-1)^{k+1}}{k!}\bbe J_k  + \frac{1}{(2p+1)!}\bbe K_{2p+1}\le {\rm Var}\; S \le\sum^{2p-1}_{k=1} 
\frac{(-1)^{k+1}}{k!} \bbe J_k - \frac{1}{(2p)!}\bbe K_{2p}.
\end{equation}
\end{thm}

\begin{proof}
The only modification compared to the proof of Theorem~\ref{jackit} is that instead of 
using the bound $\bbe J_k\ge k!\bbe R_k$ we use the fact that
\[
\bbe K_k \le k!\bbe R_k\,,
\]
which follows from the convexity of $\var^{(i_1,\dots, i_k)}$, i.e., from \eqref{ConvIne}.
\end{proof}

In particular, from Theorem~\ref{jackit} and 
Theorem~\ref{jackit2}, (the case $p=0$, being clear) 
the following inequalities hold true:
\begin{eqnarray*}
0\le \bbe K_1 \le \var\, S \le \bbe J_1, \\
0 \le \frac{1}{2} \bbe K_2 \le \bbe J_1 -  \var\, S \le\frac12\bbe J_2.
\end{eqnarray*}

\section{Relationship With The Hoeffding Decomposition}

Let us recall the notion of Hoeffding decomposition \cite{Hoe} 
(see \cite{ES} or \cite[Section2]{KR} for the general non-symmetric non-iid case). Given
a integrable random variable $f(X)$, it is the unique decomposition
\begin{eqnarray*}
f(X_1, \ldots, X_n) &=& \bbe f(X) + \sum_{1\le i\le n} h_i(X_i) 
+ \sum_{1\le i<j\le n} h_{ij} (X_i, X_j )+ \ldots\\
&=& f_0 + f_1 + \ldots + f_n
\end{eqnarray*}
such that $\bbe^{(i_s)} h_{i_1,\ldots,i_k}(X_{i_1}, \ldots, X_{i_k})=0$, 
whenever $1 \le i_1 <\ldots< i_k \le n$, $s = 1, \ldots, k$. The term $f_d$ is called the Hoeffding term of degree $d$ and these terms form an orthogonal decomposition (provided, of course, that 
$f(X)$ is square integrable); and so 
$\var f = \sum_{k=1}^n \var\,f_k=\sum_{I\subset \{1,\ldots,n\}} \bbe h_I^2$, where $I\neq \emptyset$.  

Continuing with our 
notation, for any $i=1, \dots, n$, let 
$\bbe_i$ denote the conditional expectation with respect to 
the $\sigma$-field generated  by $X_1, \dots, X_{i}$, i.e., 
$\bbe_i S:= \bbe (S\mid X_1,\dots, X_{i})$, while this time 
$\bbe_0 S = \bbe S$.    

Then, it is easily seen that above,  
$f_0 = \bbe_0 f$, $h_i = \bbe_i f - \bbe_0 f$, $i= 1, \dots, n$, 
$h_{ij} = \bbe_{ij} f - \bbe_{i} f - \bbe_{j} f +  \bbe_0 f$, $1\le i < j \le n$, etc.

The following lemma provides a relationship between the previously 
introduced iterated jackknives and the variance of the Hoeffding terms.
\begin{lem}\label{lej}
For any $k$ such that $1\le k\le n$, 
\[
\frac{1}{k!}\bbe J_k(f)=\sum_{j=k}^n{j\choose k}\var \,f_j
\], 
\[
\frac{1}{k!}\bbe K_k(f)= \var\, f_k\,, 
\]
and so 
\[
\bbe J_k(f)= \sum_{j=k}^n\frac{1}{(j-k)!}\bbe K_j(f).  
\]
\end{lem}
\begin{proof}
Rewrite the Hoeffding decomposition of $f$ as $f = \bbe f + \sum_{I\subset \{1,\ldots,n\}} h_I$, 
where again $I\neq \emptyset$.  Then, 
$\bbe^{(i)} h_I=0$, whenever $i\in I$, and $\bbe^{(i)}h_I = h_I$ otherwise.  
Hence, $\var^{(i)} h_I= \bbe^{(i)} h_I^2$, if $i\in I$ and $0$ otherwise.  Therefore, 
$ \bbe \var^{(i)} S = \sum_{i\in I} \bbe h_I^2$.

Continuing with the same reasoning, we see that $\var^{(i)}\bbe^{(j)} h_I=\bbe^{(i)} h_I^2$, 
if $i\in I$ and $j\notin I$ and $0$ otherwise, thus 
$\bbe \var^{(i)} \bbe^{(j)} S = \sum_{i\in I,j\notin I} \bbe h_I^2$, so that 
$ \bbe \var^{(i,j)} S = \sum_{\{i,j\}\subset I} \bbe h_I^2$ and by induction, we get 
that $$ \bbe \var^{(i_1,\ldots,i_k)} S = \sum_{\{i_1,\ldots,i_k\}\subset I} \bbe h_I^2\,.$$ 
If we now sum over the possible sets of indices, since each term $\bbe h_I^2$ appears 
as many times as there are subsets of 
size $k$ of $I$, this implies 
that $\bbe J_k=k!\sum_{|I|=k}^n {|I| \choose k} \bbe h_I^2 
= k!\sum_{j=k}^n{j\choose k}\var\, f_j$,  
proving the first statement.

To prove the second statement of the lemma, observe that 
$\bbe^{\overline{(i_1,\ldots,i_k)}} S = \sum_{I\subset \{i_1,\ldots,i_k\}} h_I$ so that $ \bbe \var^{(i_1,\ldots,i_k)} \bbe^{\overline{(i_1,\ldots,i_k)}} S =  \bbe h_{i_1,\ldots,i_k}^2$,  
and therefore $\bbe K_k =k!  \sum_{|I|= k}\bbe h_I^2=k!\var\, f_k$.

To obtain the third statement, just combined the previous two.  

\end{proof}

It is easily verified that \eqref{varexp} can be recovered 
from Lemma~\ref{lej} and that 
\begin{equation}\label{varexautre}
\var\; S = \bbe J_1 - \sum_{k=2}^n \frac{k-1}{k!}\bbe K_k.  
\end{equation}
Moreover, still from Lemma~\ref{lej}, 
\begin{equation}\label{varexK}
\var\; S =  \sum_{k=1}^n \frac{1}{k!}\bbe K_k.  
\end{equation}

Lemma~\ref{lej} also easily imply the 
following corollary obtained in \cite{BGS} (as part of their Theorem 1.8) which moreover can be complemented with the trivial lower bound 
$\bbe K_d/k! \le \var\, S$.

\begin{cor}\label{jackcor}  
Let $S$ have Hoeffding decomposition 
of type $S=\bbe S + \sum_{k=d}^n S_k$, i.e., such that $f_k=0$, for $1\le k< d$,
then 
\begin{equation}\label{ineq3}
{\rm Var}\; S \le \frac{1}{d!} \bbe J_{d}. 
\end{equation}
\end{cor}
\begin{proof}
 Using the fact that $f_k=0$, for $1\le k< d$, we have
 \[
 \var\,S = \sum_{j=d}^n \var\,f_j \le \sum_{j=d}^n {j\choose d} \var\,f_j  = \frac{1}{d!}\bbe J_d, 
 \]
 where the last equality follows from Lemma~\ref{lej}.
 \end{proof}

\end{document}